\documentclass[12pt, twoside]{article}
\usepackage{amsmath,amsthm}
\usepackage{amssymb,latexsym}
\usepackage{enumerate}
\usepackage[T1]{fontenc}

\newtheorem{thm}{Theorem}
\newtheorem{lem}{Lemma}

\theoremstyle{definition}
\newtheorem{ex}{Example}

\begin{document}

\title{Heat equation with a general stochastic measure on nested fractals}

\author{Vadym Radchenko\\
Kyiv National Taras Shevchenko University\\
E-mail: vradchenko@univ.kiev.ua
\and 
Martina Z\"{a}hle\\
Friedrich-Schiller-Universit\"{a}t Jena,\\
E-mail: martina.zaehle@uni-jena.de}

\date{}

\maketitle


\renewcommand{\thefootnote}{}



\renewcommand{\thefootnote}{\arabic{footnote}}
\setcounter{footnote}{0}


\begin{abstract}
A stochastic heat equation on an unbounded nested fractal driven by a general stochastic measure is investigated. Existence,
uniqueness and continuity of the mild solution are proved provided that the spectral dimension of the fractal is less
than~4/3.
\end{abstract}

\section{Introduction}
\label{scintr}

The main object of the paper is the stochastic heat equation on an unbounded nested fractal $\tilde{E}$ which can
formally be written as
\begin{equation}
\label{eqhshfx} du(t, x)=\mathcal{L} u(t, x)\,dt +f(t, x, u(t, x))\,dt+ \sigma(t, x)\,d{\mu}(x),\quad u(0, x)=u_0(x)\,
,
\end{equation}
where $u$ is an unknown measurable stochastic function, $(t, x)\in [0,\ T]\times \tilde{E}$, $\mu$ is a general
stochastic measure, $\mathcal{L}$ is the infinitesimal generator of Brownian motion on $\tilde{E}$. We consider the
equation in the mild sense (see (\ref{eqhsif}) below) and obtain existence, uniqueness and continuity of the solution.
The Euclidean version on the real line was proved in~\cite{rads09}.

Here we use the Brownian motion on $\tilde{E}$ constructed in~\cite{kuma93}. Our results are based on the H\"{o}lder
continuity and an upper estimate of the transition density of the process derived in that paper. Diffusions on nested
fractals also were constructed in~\cite{barl98,krebs91,lind90}.

SPDE on fractals and general metric measure spaces are studied e.g. in~\cite{prerok} (with an infinite-dimensional Wiener
process as integrator) and in~\cite{hinzae} (with pathwise definition of the stochastic integral). In our paper we
consider the case of an additive noise, but with a rather general stochastic integrator. We do not assume path
regularity or moment existence for the stochastic measures.

The paper is organized as follows. Section~\ref{scprel} summarizes some basic facts about nested fractals and stochastic
measures. In Section~\ref{scpari} we obtain the continuity of paths of stochastic parameter integrals on these fractals.
Section~\ref{scsheq} contains the main result of the paper.

\section{Preliminaries}
\label{scprel}

\subsection{Heat kernel on nested fractals}
\label{sshksg}

Let $E$ be a {\it nested fractal} (see Definition~1.3~\cite{kuma93}). Then $E\subset\mathbb{R}^d$ is a compact set, $E$ has a
self-similar structure, i.e., $E=\bigcup_{i=1}^{N}\psi_i(E)$, where $\psi_i:\mathbb{R}^d\to\mathbb{R}^d$ are similitudes with
the same contraction factor, 
\[
\bigl|\psi_i(x)-\psi_i(y)\bigr|=\alpha^{-1}|x-y|,\ \alpha>1.
\]
In the following we assume that $\psi_1(x)=\alpha^{-1}x$.

Let us denote  by $d_s$ the {\it spectral dimension}, by $d_w$ the {\it walk dimension}, and by $d_f$ the {\it Hausdorff
dimension} of the fractal set $E$. Recall that $d_f=\log N / \log\alpha$, $d_w=2d_f / d_s$.

Let $F$ be the set of fixed points of $\psi_i$, $1\le i\le N$. $x\in F$ is called an {\it essential fixed point} if there
exist $y\in F,\ j\ne k$ such that $\psi_j(x)=\psi_k(y)$. Let $F^{(0)}$ be the set of essential fixed points,
\[
F^{(n)}=\bigcup_{i_1,\dots,i_n} \psi_{i_1}\circ\dots \circ\psi_{i_n}\bigl(F^{(0)}\bigr).
\]
Further, let us define the unbounded nested fractal~$\tilde{E}$. We set
\[
\tilde{F}^{(0)}=\bigcup_{n\ge 0}\bigl(\alpha^n F^{(n)}\bigr),\quad \tilde{F}^{(n)}=\alpha^{-n} \tilde{F}^{(0)},\quad
\tilde{E}=\textrm{Cl}\,\Bigl(\bigcup_{n\in\mathbb{Z}} \tilde{F}^{(n)} \Bigr)\, .
\]
Note that
\begin{equation}
\label{eqdfte} \tilde{E}=\bigcup_{n\ge 1}\bigcup_{i_1,\dots,i_n} \bigl( \alpha^n
\psi_{i_1}\circ\dots \circ\psi_{i_n} (E)\bigr),
\end{equation}
where any two sets in the union either coincide or have finite intersection.

Let ${\sf m}$ be the $d_f$-{\it dimensional Hausdorff measure} on $\tilde{E}$ such that ${\sf m}(E)=1$, $\mathcal{B}(\tilde{E})$ denote
the Borel $\sigma$-algebra on $\tilde{E}$.

\textbf{Assumption 1.} There exists $k\in\mathbb{N}$ satisfying the following. If for $x,\ y\in E$ we have $|x-y|\le
\alpha^{-m}$, then there exist $x_1$, $x_2$, \dots, $x_l$ ($l\le k$) such that $x_1=x$, $x_l=y$, $x_2,\ \dots,\
x_{l-1}\in F^{(m)}$ and $x_{j}$, $x_{j+1}$, $1\le j\le l-1$ lie in the same set $\psi_{i_1}\circ\dots
\circ\psi_{i_m}\bigl(E\bigr)$.

Let $p(t,x,y)$, $t>0$, $x,\ y\in \tilde{E}$, be the transition density of Brownian motion on $\tilde{E}$ constructed
in~\cite{kuma93} provided that Assumption 1 holds. Then $p$ is symmetric in $x$, $y$, is jointly continuous in
$(t,x,y)$ and satisfies the following {\it H\"older continuity and sub-Gaussian estimate} for some constants $c_1, c_2,
c_3>0$, and $d_J>1$.

\begin{thm} \cite[Theorems 5.1 and 5.2]{kuma93}\label{thbarp}

a) $\bigl|p(t, x, y_1)-p(t, x, y_2)\bigr|\le c_1 t^{-1} \bigl|y_1-y_2\bigr|^{d_w-d_f}$.

b) $p(t, x, y)\le c_2 t^{-d_s/2}\exp\bigl\{-c_3\bigl(|x-y|^{d_w} / t \bigr)^{1/(d_J-1)}\bigr\} \, .$

\end{thm}

We will take this $p(t, x, y)$ as the heat kernel for the stochastic heat equation in the mild form~(\ref{eqhsif}). Note
that $p$ fulfills the deterministic heat equation $\partial p / \partial t=\mathcal{L} p$, where $\mathcal{L}$ is the
infinitesimal generator of the Brownian motion on $\tilde{E}$ \cite[Theorem~6.3]{kuma93}.

The same estimates of the heat kernel on the Sierpi\'{n}ski gasket, both the bounded and unbounded cases, are given in
Theorem~2.23~\cite{barl98}, and the inequalities for the two cases are equivalent. In our paper, we consider
unbounded nested fractals, but we expect the same results for the  bounded versions.

\begin{ex} The \emph{Vicsek set} is the main example of this paper. Let $H_0$ be the closed unit square with vertices
$a_1=(0,0)$, $a_2=(0,1)$, $a_3=(1,1)$, $a_4=(1,0)$. We take $a_5=(1/2,1/2)$ and set
\[
\psi_{i}(x)=a_i+(x-a_i)/3,\quad 1\le i\le 5,\quad H_{n+1}=\bigcup_{i=1}^{5}\psi_i(H_{n}),\quad E_{VS}=\bigcap_{n\ge
0}H_n\, .
\]
Then $E_{VS}$ is called a Vicsek set (or Vicsek snowflake).

It is known that $E_{VS}$ is a nested fractal \cite[Proposition 2.1]{krebs91}, and we have for this set $d_f={\log
5} / {\log 3}$, $d_w={\log 15}/{\log 3}$ \cite[Section 2]{barl98},  $d_s={\log 25}/{\log 15}$ (see calculations in~\cite[Section~6]{zhou09} for $n=2$). Note that Assumption 1 holds for this set with $k=4$.
\end{ex}

\subsection{Stochastic measures}
\label{ssstme}

Let $L_0=L_0(\Omega,\ \mathcal{F},\ {\sf P})$ be the set of all real-valued random variables defined on the
probability space $(\Omega,\ \mathcal{F},\ {\sf P})$ (more precisely, the set of equivalence classes).
Convergence in $L_0$ means the convergence in probability. Let ${\sf X}$ be an arbitrary set and
${\mathcal{B}}$ be a $\sigma$-algebra of subsets of ${\sf X}$.

\medskip
\textbf{Definition.} {\em A $\sigma$-additive mapping $\mu:\ {\mathcal{B}}\to L_0$ is called  stochastic
measure (SM).}
\medskip

We do not assume any martingale or moment conditions for $\mu$. Examples of SM are the following. For square integrable
martingale $X(t),\ 0\le t\le T,$ set $\mu(A)=\int_{[0,\ T]} {\bf 1}_A(t)\, dX(t)$, then $\mu$ is a SM on the Borel
$\sigma$-algebra $\mathcal{B}\bigl([0,\ T]\bigr)$. For unconditionally convergent in probability series $\sum_{n\ge
1}\xi_n$ and $\{x_n\}\subset {\sf X}$ set function $\mu(A)=\sum_{n\ge 1}\xi_n {\bf 1}_A(x_n)$ is a SM on~$2^{\sf X}$.

\begin{ex} Let $E_{VS}$ be the Vicsek set. We construct an example of a SM on the Borel $\sigma$-algebra
$\mathcal{B}\bigl(E_{VS}\bigr)$. Let $\mu$ be a SM on $\mathcal{B}((0,\ 1])$ such that $\mu(\{x\})=0$ a.~s. for each
$x\in (0,\ 1]$. For $n\ge 0$ we take a representation
\[
E_{VS}=\bigcup_{1\le k\le 5^n} E^{(n)}_k,\quad E^{(n)}_k=\psi_{i_1}\circ\dots
\circ\psi_{i_n}\bigl(E_{VS}\bigr)=E^{(n+1)}_{5k-4}\cup\dots\cup E^{(n+1)}_{5k}\, .
\]
We set 
\[
\mu_E\bigl(E^{(n)}_k\bigr)=\mu\bigl(\bigl((k-1)5^{-n},\ k5^{-n}\bigr]\bigr),\quad \mu_E(\{x\})=0
\]
for each $x\in E^{(n)}_k\cap E^{(n)}_i$ and define $\mu_E$ on the generated algebra by additivity. Then $\mu_E$ is continuous in
$\emptyset$ and the values of $\mu_E$ are bounded in probability (see \cite[Theorem B.2.1]{kwawoy}). By Theorem 1
of~\cite{radtvp} we can extend the stochastic set function $\mu_E$ to $\mathcal{B}(E_{VS})$. Taking the weighted sum
of stochastic measures defined on the sets \mbox{$\alpha^n \psi_{i_1}\circ\dots \circ\psi_{i_n} \bigl(E_{VS}\bigr)$}
from~(\ref{eqdfte}), we can easily extend this stochastic measure to the unbounded Vicsek set $\tilde{E}_{VS}$.
\end{ex}

In \cite{kwawoy} the integral of measurable deterministic functions $f:{\sf X}\to{\mathbb R}$ with respect to (w.r.t.)
SM is constructed. The dominated convergence theorem holds for this integral \cite[Proposition 7.1.1]{kwawoy}, and all
bounded measurable functions are integrable. We will use the following statement.

\begin{lem} \cite[Lemma 3.1]{rads09}
\label{lmfkmu} Let $g_l:\ {\sf X}\to {\mathbb R},\ l\ge 1$, be measurable functions such that 
$\bar{g}(x)=\sum_{l=1}^{\infty} \bigl|{g_l}(x)\bigr|$ is integrable w.r.t.~$\mu$. Then
$\sum_{l=1}^{\infty}\Bigl(\int_{\sf X} g_l\,d\mu \Bigr)^2<\infty$~a.~s.
\end{lem}

\section{Parameter stochastic integrals on nested fractal}
\label{scpari}

Let $\tilde{E}$ be an unbounded nested fractal, $\mu$ be a SM defined on $\mathcal{B}(\tilde{E})$,
 ${\sf Z}$ be a metric space, function $h:{\sf Z}\times \tilde{E}\to \mathbb{R}$ be integrable w.r.t. $\mu$ on $\tilde{E}$ for each first argument
$z\in {\sf Z}$. Then the stochastic function
\[
\eta(z)=\int_{\tilde{E}} h(z, y)\,d\mu(y),\quad z\in {\sf Z}
\]
is determined.

\begin{thm}\label{thhcms}
Suppose that $h$ is bounded  and continuous on ${\sf Z}\times \tilde{E}$, for some $K_h>0$, $\beta(h)>{d_f}/{2}$ we
have
\begin{equation}\label{eqmscf}
\bigl|h(z, y_1)-h(z, y_2)\bigr|\le K_h \bigl|y_1- y_2\bigr|^{\beta(h)}\, ,
\end{equation}
and the SM $\mu$ is such that $\mu(\{x\})=0$ a.~s. for each $x\in \tilde{E}$, and for some $\tau>d_f/2$ the function
$|y|^{\tau}$ is integrable w.r.t. $\mu$ on $\tilde{E}$.\\ Then $\eta(z)$ has a version with continuous paths.
\end{thm}

\begin{proof} Using~(\ref{eqdfte}), we can take a representation $\tilde{E}=\bigcup_{j\ge 1} E_j$, where each $E_j$ is a
bounded part of the form 
\[
\alpha^n \psi_{i_1}\circ\dots \circ\psi_{i_n} (E),\quad E_j\subset \alpha^n E,\quad
E_j\not\subset \alpha^{n-1} E\ \textrm{for}\ N^{n-1}<j\le N^n. 
\]
Let 

\[
\gamma=\inf\{|y|,\ y\in E\setminus \alpha^{-1} E\},
\]
we have $\gamma>0$ because $(E\setminus \alpha^{-1} E)\subset \bigcup_{i=2}^{N}\psi_i(E)$, and $0\notin \psi_i(E)$, $2\le
i\le N$ by~\cite[Lemma 5.23]{barl98}. Recalling that $d_f=\log N/\log a$ we infer
\begin{equation}\label{eqesmy}
j\le (|y|/\gamma)^{d_f} ,\quad y\in E_j,\quad j\ge 2.
\end{equation}

Consider $\eta_j(z)=\int_{E_j} h(z, y)\,d\mu(y)$, $j\ge 1$. For each $n\ge 0$ we can take the representation
\[
E_j=\bigcup_{1\le k\le N^n} E^{(nj)}_k,\quad E^{(nj)}_k=\psi_{i_1}\circ\dots \circ\psi_{i_n}\bigl(E_j\bigr)\, .
\]
For
\[
S^{(nj)}(z)=\sum_{1\le k\le N^n}h\bigl(z,
y^{(nj)}_k\bigr)\mu\bigl(E^{(nj)}_k\bigr),\quad y^{(nj)}_k\in E^{(nj)}_k,
\]
the dominated convergence theorem \cite[Proposition 7.1.1]{kwawoy} implies
that $ S^{(nj)}(z)\stackrel{{\sf P}}{\to} \eta_j(z)$, $n\to\infty$, for each $z\in{\sf Z}$.

For convenience, assume that $\textrm{diam}\, E=1$, then $\textrm{diam}\, E^{(nj)}_k=\alpha^{-n}$. If
$0<\beta<\beta(h)-{d_f}/{2}$, using~(\ref{eqmscf}) and the Cauchy-Schwarz inequality,
 we get
\begin{eqnarray}
\nonumber \bigl|\eta_j(z)\bigr|\le \bigl|S^{(0j)}(z)\bigr|+\sum_{n\ge 0}\bigl|S^{((n+1)j)}(z)-S^{(nj)}(z)\bigr|\le
\bigl|S^{(0j)}(z)\bigr|\\
+\sum_{n\ge 0}\sum_{1\le k\le N^{n+1}} \bigl|h\bigl(z, y^{((n+1)j)}_k\bigr)-h\bigl(z,
y^{(nj)}_{k'}\bigr)\bigr|\bigl|\mu\bigl(E^{((n+1)j)}_{k}\bigr)\bigr|\le
\nonumber \bigl|S^{(0j)}(z)\bigr|\\
\nonumber +K_h\sum_{n\ge 0}\sum_{1\le k\le N^{n+1}} {\alpha}^{-n\beta(h)}\bigl|\mu\bigl(E^{((n+1)j)}_{k}\bigr)\bigr|\le
\bigl|h\bigl(z, y^{(0j)}_1\bigr)\bigr| \bigl|\mu(E_j)\bigr| \\
\label{eqeset} +K_h\Bigl(\sum_{n\ge 0} N^{n+1} {\alpha}^{2n(\beta-\beta(h))}\Bigr)^{1/2}\Bigl(\sum_{n\ge 0}
{\alpha}^{-2n\beta}\sum_{1\le k\le N^{n+1}}\bigr|\mu\bigl(E^{((n+1)j)}_{k}\bigr)\bigr|^2\Bigr)^{1/2}\, .
\end{eqnarray}
(The number $k'$ is chosen such that $E^{((n+1)j)}_{k}\subset E^{(nj)}_{k'}$.) Lemma~\ref{lmfkmu} implies
\[
\sum_{n\ge 0} {\alpha}^{-2n\beta}\sum_{1\le k\le N^{n+1}}\bigl|\mu\bigl(E^{((n+1)j)}_{k}\bigr)\bigr|^2<\infty\quad
\mbox{\textrm ~a.~s.}
\]
Note that $2(\beta(h)-\beta)>d_f=\log N / \log\alpha$. As in~(\ref{eqeset}) we obtain that
\[
\sum_{n\ge m}\bigl|S^{((n+1)j)}(z)-S^{(nj)}(z)\bigr|\le
K_h\Bigl(\sum_{n\ge m} N^{n+1} {\alpha}^{2n(\beta-\beta(h))}\Bigr)^{1/2}\Bigl(\sum_{n\ge m}
{\alpha}^{-2n\beta}\sum_{1\le k\le N^{n+1}}\bigr|\mu\bigl(E^{((n+1)j)}_{k}\bigr)\bigr|^2\Bigr)^{1/2},
\]
consequently, the series 
\[
\sum_{n\ge 0}\bigl|S^{((n+1)j)}(z)-S^{(nj)}(z)\bigr|
\]
converges uniformly in $z$ a.s. Each $S^{(nj)}(z)$
has continuous paths, therefore $\eta_j(z)$ has a continuous version.

 Recall that
$\eta(z)=\sum_{j\ge 1}\eta_j (z)$. Using $\left|h\left(z,\ y^{(0j)}_1\right)\right|\le M$ for some  constant $M$, estimate~(\ref{eqeset}) and the Cauchy-Schwarz inequality we get for $\rho>1/2$,

\begin{eqnarray}
\nonumber \sum_{j\ge m}\left|\eta_j (z)\right|\le \left(\sum_{j\ge m}j^{-2\rho}\right)^{1/2}\left(\sum_{j\ge
m}j^{2\rho}M^2 \left|\mu\left(E_j\right)\right|^2\right)^{1/2}\\
\label{eqserj} + K_h\left(\sum_{j\ge m}j^{-2\rho}\right)^{1/2} \left(\sum_{j\ge m} \sum_{n\ge 0} \sum_{1\le k\le
N^{n+1}} j^{2\rho} {\alpha}^{-2n\beta}\left|\mu\left(E^{((n+1)j)}_{k}\right)\right|^2 \right)^{1/2}\, .
\end{eqnarray}
Now we apply Lemma~\ref{lmfkmu} to the function sequences
\begin{eqnarray*}
\left\{g_l(y),\ l\ge 1\right\}=\left\{j^{\rho}M {\bf 1}_{E_j}(y),\ j\ge 1\right\},\\
\left\{g_l(y),\ l\ge 1\right\}=\left\{j^{\rho}{\alpha}^{-n\beta} {\bf 1}_{E^{((n+1)j)}_{k}}(y),\ j\ge 1,\ n\ge 0,\ 1\le
k\le N^{n+1}\right\}.
\end{eqnarray*}
where $\sum_{l\ge 1} \left|g_l(y)\right|\le C|y|^{\rho d_f}$ (see~(\ref{eqesmy})). From the integrability assumption of
our Lemma we infer that the all series in~(\ref{eqserj}) converge for $m=1$. Thus we obtain that $\sum_{j\ge
m}\left|\eta_j (z)\right|$ tends to 0 uniformly a.s. Therefore, $\eta(z)$ has a continuous version.
\end{proof}

\section{Stochastic heat equation on a nested fractal}
\label{scsheq}

We consider the heat equation on~$\tilde{E}$ in the following mild sense
\begin{eqnarray}
\nonumber u(t, x)=\int_{\tilde{E}}{p}(t, x, y)u_0(y)\,d{\sf m}(y)
+\int_0^t ds \int_{\tilde{E}}{p}(t-s, x, y)f(s, y, u(s, y))\,d{\sf m}(y)\\
+\int_{\tilde{E}} d\mu(y)\int_0^t  {p}(t-s, x, y)\sigma(s, y)\,ds,\quad t\in [0,\ T],\quad x\in \tilde{E}\, ,
\label{eqhsif}
\end{eqnarray}
where $u:\ [0,\ T]\times \tilde{E}\times\Omega\to\mathbb{R}$ is an unknown measurable stochastic function, $p$ is the
transition density of the Brownian motion on $\tilde{E}$ constructed in~\cite{kuma93}, $\mu$ is a SM on
$\mathcal{B}(\tilde{E})$ and $u_0$ is a random initial function.

We make the following assumptions.

\textbf{Assumption 2.} $u_0:{\tilde{E}}\times\Omega\to{\mathbb R}$ is measurable and for each
$\omega\in \Omega$ the function  $u_0(\cdot, \omega)$ is continuous and bounded on  $\tilde{E}$, $\bigl|u_0(y,\
\omega)\bigr|\le C_{u_0}(\omega)$, and vanishes at $\infty$.

\textbf{Assumption 3.} $f:[0, T]\times{\tilde{E}}\times{\mathbb R}\to{\mathbb R}$ is measurable and bounded,
$|f(s, y, r)|\le C_f$.

\textbf{Assumption 4.} $f(s, y, r)$ is uniformly Lipschitz in $y\in {\tilde{E}},\ r \in{\mathbb R}$,
\[
\bigl|f(s, y_1, r_1)-f(s, y_2, r_2)\bigr|\le K_f \bigl(\bigl|y_1- y_2\bigr| + \bigl|r_1-r_2\bigr|\bigr)\, .
\]

\textbf{Assumption 5.} $\sigma:[0, T]\times{\tilde{E}}\to{\mathbb R}$ is measurable and bounded, $\bigl|\sigma(s,
y)\bigr|\le C_{\sigma}$.

\textbf{Assumption 6.} $\sigma(s, y)$ is uniformly H\"{o}lder continuous in $y\in \tilde{E}$,
\[
\bigl|\sigma(s, y_1)-\sigma(s, y_2)\bigr|\le K_{\sigma}\bigl|y_1- y_2\bigr|^{\beta(\sigma)},\quad \beta(\sigma)>d_f/2.
\]

\textbf{Assumption 7.} The spectral dimension $d_s$ of $E$ is less than $4/3$.

\textbf{Assumption 8.} $\mu\bigl(\{x\}\bigr)=0$ a.~s. for each $x\in \tilde{E}$.

\begin{lem}\label{lmcnth}
Under the Assumptions 1, 5, 6 and 7 the function
\[
h(z, y)=\int_0^t  {p}(t-s, x, y)\sigma(s, y)\,ds,\quad z=(t, x)\in [0,\ T]\times \tilde{E},\ y\in \tilde{E},
\]
satisfies the conditions of Theorem~\ref{thhcms} for  ${\sf Z}=[0,\ T]\times \tilde{E}$.
\end{lem}

\begin{proof} We  check the H\"{o}lder continuity of $h$, the other conditions obviously hold true. By $C$ we will denote a positive
constant whose values may change. Using Assumption 6 and Theorem~\ref{thbarp}~b) we get
\begin{eqnarray*}
\bigl|h(z, y_1)-h(z, y_2)\bigr|\le
\int_0^t  \bigl|{p}(t-s, x, y_1)-{p}(t-s, x, y_2)\bigr| \bigl|\sigma(s, y_1)\bigr|\,ds \\
+\int_0^t  {p}(t-s, x, y_2)\bigl|\sigma(s, y_1)-\sigma(s, y_2)\bigr|\,ds\\
\le C_\sigma \int_0^t \bigl|{p}(t-s, x, y_1)-{p}(t-s, x, y_2)\bigr|\,ds+C K_{\sigma}\bigl|y_1-
y_2\bigr|^{\beta(\sigma)}\, .
\end{eqnarray*}

By Theorem~\ref{thbarp}~a), for the first term we have
\begin{eqnarray*}
\int_0^t \bigl|{p}(t-s, x, y_1)-{p}(t-s, x, y_2)\bigr|\,ds= \int_0^{t-\delta}+\int_{t-\delta}^t\\
\le C\int_0^{t-\delta} (t-s)^{-1} \bigl|y_1- y_2\bigr|^{d_w-d_f}\,ds+ C\int_{t-\delta}^t (t-s)^{-d_s/2}\,ds\\
\le  C \bigl(\log T - \log \delta \bigr) \bigl|y_1- y_2\bigr|^{d_w-d_f}+C\delta^{1-d_s/2}.
\end{eqnarray*}
Taking $\delta=\bigl|y_1- y_2\bigr|^{(d_w-d_f)/(1-d_s/2)}$ and using that for all $\gamma>0$, $t^\gamma\log t\to 0$ as
$t\to 0$, we obtain (\ref{eqmscf}) for any $\beta(h)<\min \{d_w-d_f,\ \beta(\sigma)\}$. If $d_s={2d_f} /
{d_w}<{4}/{3}$ we can find $\beta(h)>{d_f}/{2}$.
\end{proof}

\begin{thm}
Suppose Assumptions 1--8 hold.

1) Equation~(\ref{eqhsif}) has a solution $u(t, x)$. If $v(t, x)$ is another solution, then for each
$t$ and $x$, $u(t, x)=v(t, x)$~a.~s.

2) If, additionally, the function $|y|^{\tau}$ be integrable w.r.t. $\mu$ on~$\tilde{E}$ for some $\tau>d_f/2$, then the solution $u(t,
x)$ has a version continuous on $[0,\ T]\times \tilde{E}$.
\end{thm}

\begin{proof} Take $u^{(0)}(t, x)=0$ and set
\begin{eqnarray}
\nonumber
u^{(n+1)}(t, x)=\int_{\tilde{E}}{p}(t, x, y)u_0(y)\,d{\sf m}(y)
+\int_0^t ds \int_{\tilde{E}}{p}(t-s, x, y) f(s, y, u^{(n)}(s, y))\,d{\sf m}(y)\\
+\int_{\tilde{E}} d\mu(y) \int_0^t {p}(t-s, x, y)\sigma(s, y)\,ds,\quad n\ge 0\, . \label{mliter}
\end{eqnarray}
The measurability of the first and second summands of~(\ref{mliter}) follows from the Fubini theorem. The third summand
may be rewritten as limit in probability of the integrals of the simple functions
\[
\int_{E}h_i(t, x, y)\,d\mu(y)\, ,\quad i\to\infty,\quad h_i(t, x, y)=\sum_{k=1}^{l^{(i)}} h_k^{(i)}(t, x){\bf
1}_{A_k^{(i)}}(y).
\]
By~\cite[Lemma~2]{radt96}, the limit in probability of measurable processes is again measurable.

Further, for all $n,\ t,\ x$ we take the same version of the stochastic integral and obtain the following estimates for
all $\omega\in\Omega$. Assumption~5 implies
\begin{eqnarray}
\nonumber
\bigl|u^{(n+1)}(t, x)-u^{(n)}(t, x)\bigr|\\
\nonumber
=\Bigl|\int_0^t ds \int_{\tilde{E}}{p}(t-s, x, y)\bigl[f(s, y, u^{(n)}(s, y))-f(s, y, u^{(n-1)}(s, y))\bigr]\,d{\sf m}(y)\Bigr|\\
\le K_f\int_0^t ds \int_{\tilde{E}}{p}(t-s, x, y)\bigl|u^{(n)}(s, y)- u^{(n-1)}(s, y)\bigr|\,d{\sf m}(y),\quad n\ge 2.
\label{equdif}
\end{eqnarray}
Since $f$ is bounded,we infer
\[
\bigl|u^{(2)}(t, x)-u^{(1)}(t, x)\bigr| \le 2C_f\int_0^t ds \int_{\tilde{E}}{p}(t-s, x, y)\,d{\sf m}(y)=2C_f t\, .
\]

Set $g_n(t)=\sup_{x\in {\tilde{E}}}\bigl|u^{(n+1)}(t, x)-u^{(n)}(t, x)\bigr|,\quad n\ge 1\, .$ Then we get
from~(\ref{equdif}),
\begin{equation}
\label{eqestg} g_n(t)\le K_f\int_0^t g_{n-1}(s)\, ds\, ,\quad{ hence}\quad g_n(t)\le 2C_f K_f^n\frac{t^{n+1}}{(n+1)!} \,
,
\end{equation}
i.e., the series $\sum_{n=0}^{\infty} g_n(t)$ converges uniformly  on $[0, T]$. We define $u(t, x):=\lim_{n\to\infty}
u^{(n)}(t, x)$. Taking in~(\ref{mliter}) the limit as $n\to\infty$, we arrive at~(\ref{eqhsif}).

If $u(t, x)$ and $v(t, x)$ both are solutions to~(\ref{eqhsif}), then we may repeat the arguments from~(\ref{equdif})
to~(\ref{eqestg})  for $u-v$ (instead of $u^{(n+1)}-u^{(n)}$) and $g(t)=\sup_{x\in {\tilde{E}}}|u(t, x)-v(t, x)|$
(instead of $g_n$) and obtain $g=0$.

From Lemma~\ref{lmcnth} and Theorem~\ref{thhcms} it follows that $u^{(n)}$ has a continuous version. Then $u(t, x)$ as
uniform limit of $u^{(n)}(t, x)$ is continuous, too. \end{proof}

\textbf{Remark.} For the Vicsek set $E_{VS}$ we have $d_s={\log 25}/{\log 15}<{4}/{3}$, therefore $E_{VS}$ satisfies
Assumption~7. It is known that Sierpi\'{n}ski gasket is a nested fractal with $d_s={\log 9}/{\log 5}>{4}/{3}$, and
Assumption~7 fails for this set.

\section*{Acknowledgments}

V.Radchenko acknowledges the support of Alexander von Humboldt Foundation, grant 1074615, and thanks
Friedrich-Schiller-Universit\"{a}t Jena for hospitality.

\bibliographystyle{plain}
\bibliography{RadchenkoZaehleFrStMBib}

\begin{thebibliography}{10}

\bibitem{barl98}
M.~T. Barlow.
\newblock {\em Diffusions on fractals}, volume 1690 of {\em Lect. Notes Math.},
  pages 1--121.
\newblock Springer, Berlin / Heidelberg, 1998.

\bibitem{hinzae}
M.~Hinz and M.~Z\"{a}hle.
\newblock Semigroups, potential spaces and applications to ({S}){P}{D}{E}.
\newblock {\em Potent. Anal.}, pages 1--33, 2011.

\bibitem{krebs91}
W.~Krebs.
\newblock A diffusion defined on on a fractal state space.
\newblock {\em Stoch. Proc. Appl.}, 37:199--212, 1991.

\bibitem{kuma93}
T.~Kumagai.
\newblock Estimates of transition densities for {B}rownian motion on nested
  fractals.
\newblock {\em Probab. Theory Relat. Fields}, 96:205--224, 1993.

\bibitem{kwawoy}
S.~Kwapie\'{n} and W.~A. Woyczy\'{n}ski.
\newblock {\em Random Series and Stochastic Integrals: Single and Multiple}.
\newblock Birkh\"{a}user, Boston, 1992.

\bibitem{lind90}
T.~Lindstr{\o}m.
\newblock {\em Brownian Motion on Nested Fractals}, volume 420 of {\em Mem.
  AMS}.
\newblock AMS, 1990.

\bibitem{prerok}
C.~Pr\'{e}v\^{o}t and M.~R\"{o}ckner.
\newblock {\em A Concise Course on Stochastic Partial Differential Equations},
  volume 1905 of {\em Lect. Notes Math.}
\newblock Springer, Berlin / Heidelberg, 2007.

\bibitem{radtvp}
V.~Radchenko.
\newblock Integrals with respect to random measures and random linear
  functionals.
\newblock {\em Theory Probab. Appl.}, 36:621--623, 1991.

\bibitem{radt96}
V.~Radchenko.
\newblock On a definition of the integral of a random function.
\newblock {\em Theory Probab. Appl.}, 41:597--601, 1997.

\bibitem{rads09}
V.~Radchenko.
\newblock Mild solution of the heat equation with a general stochastic measure.
\newblock {\em Studia Mathematica}, 194:231--251, 2009.

\bibitem{zhou09}
D.~Zhou.
\newblock Spectral analysis of {L}aplacians on the {V}icsek set.
\newblock {\em Pacific J. Math.}, 24:369--398, 2009.

\end{thebibliography}

\end{document}